\documentclass{amsart}
\usepackage[pagebackref]{hyper ref}
\usepackage{graphicx}
\usepackage{verbatim}
\usepackage{mathtools}
\usepackage{tensor}
\usepackage{tikz,tikz-cd}

\newtheorem{thm}{Theorem}[section]
\newtheorem{cor}[thm]{Corollary}
\newtheorem{lem}[thm]{Lemma}
\newtheorem{prop}[thm]{Proposition}

\theoremstyle{definition}
\newtheorem{defn}[thm]{Definition}
\theoremstyle{remark}
\newtheorem{rem}[thm]{Remark}

\numberwithin{equation}{section}

\newcommand{\set}[1]{\left\{#1\right\}}
\newcommand{\Real}{\mathbb R}

\newcommand{\func}[1]{\ensuremath{\mathop{\mathrm{ #1}}} \nolimits}

\newcommand{\xX}[0]{\mathbf{x}}

\newcommand{\eE}[0]{\mathbf{e}}

\newcommand{\TT}[0]{\mathbf{T}}

\newcommand{\yY}[0]{\mathbf{y}}

\renewcommand{\SS}{\mathbb{S}^1}

\newcommand{\thetanabla}{\tensor[^\theta]{{\nabla}}{}}

\newcommand{\taunabla}{\tensor[^\tau]{{\nabla}}{}}
\newcommand{\RP}{\mathbb{RP}}

\newcommand{\Mob}[0]{\func{M\ddot{o}b}}
\newcommand{\Diff}[0]{\func{Diff}}
\newcommand{\BDiff}[0]{\func{BDiff}}
\newcommand{\HDiff}[0]{\func{HDiff}}

\begin{document}
\title[The ovals of Benguria and Loss]{One-dimensional projective structures, convex curves and the Ovals of Benguria \& Loss}
\date{August 26, 2014}
\author[J.~Bernstein and T.~Mettler]{Jacob Bernstein and Thomas Mettler}
\address{Department of Mathematics, Johns Hopkins University, Baltimore, MD 21218, USA}
\email{bernstein@math.jhu.edu}
\address{Department of Mathematics, ETH Z\"urich, 8092 Z\"urich, Switzerland}
\email{mettler@math.ch}
\thanks{The first author was partially supported by the NSF Grant DMS-1307953. The second author was supported by Schweizerischer Nationalfonds SNF via the postdoctoral fellowship PA00P2\_142053.}
\begin{abstract}
 Benguria and Loss have conjectured that, amongst all smooth closed curves in $\Real^2$ of length $2\pi$, the lowest possible eigenvalue 
of the operator $L=-\Delta+\kappa^2$ is $1$.  They observed that this value was achieved on a 
two-parameter family, $\mathcal{O}$, of geometrically distinct ovals containing the round circle and collapsing to a multiplicity-two 
line 
segment.  We characterize the curves in $\mathcal{O}$ as absolute minima of two related geometric functionals.  We also 
discuss a connection with
projective differential geometry and use it to explain the natural symmetries of all three problems.
\end{abstract}
\maketitle

\section{Introduction}
In \cite{MR2091490}, Benguria and Loss conjectured that for any, $\sigma$, a smooth closed curve in $\Real^2$ of length $2\pi$, the 
lowest eigenvalue, $\lambda_\sigma$, of the operator $L_\sigma=-\Delta_\sigma+\kappa^2_\sigma$ satisfied $\lambda_\sigma\geq 1$. 
That is, they 
conjectured that for all such $\sigma$ and all functions $f\in H^1(\sigma)$,
\begin{equation} \label{OvalsSpecEqn}
 \int_{\sigma} |\nabla_\sigma f|^2 +\kappa^2_\sigma f^2 \; \mathrm{d}s \geq \int_{\sigma} f^2 \; \mathrm{d}s,
\end{equation}
where $\nabla_\sigma f$ is the intrinsic gradient of $f$, $\kappa_\sigma$ is the geodesic curvature and $\mathrm{d}s$ is the length element.
This conjecture was motivated by their observation that it was equivalent to a certain one-dimensional Lieb-Thirring inequality 
 with conjectured sharp constant.  They further observed that the above inequality is saturated on 
a two-parameter family of strictly convex curves which contains the round circle and degenerates into a multiplicity-two line segment.  The 
curves in this 
family look like ovals and so we call them the ovals of Benguria and Loss and denote the family by $\mathcal{O}$.  Finally, 
they showed that for closed curves $\lambda_\sigma\geq \frac{1}{2}$.

Further work on the conjecture was carried out by Burchard and Thomas in \cite{MR2203162}. They showed that $\lambda_\sigma$ is 
strictly minimized 
in a certain neighborhood of $\mathcal{O}$ in the space of closed curves -- verifying the conjecture in this 
neighborhood.  More globally, Linde~\cite{Linde} improved 
the lower bound to $\lambda_\sigma\geq 0.608$ when $\sigma$ is a planar convex curves.  In addition, he showed that 
$\lambda_\sigma\geq 1$ when $\sigma$ satisfied a certain symmetry condition.   Recently, Denzler~\cite{Denzler2013} has shown that if 
the conjecture is false,  then the 
infimum of $\lambda_\sigma$ over the space of closed curves is achieved by a closed strictly convex planar 
curve.  Coupled with Linde's work, this implies that for any closed 
curve 
$\lambda_\sigma\geq 0.608$.  In a different direction, 
the first author and Breiner in \cite{MR3148113} connected the conjecture to a certain convexity property for the length of 
curves in 
a minimal annulus. 

In the present article, we consider the family $\mathcal{O}$ and observe that the curves in this class are the unique minimizers of two 
natural 
geometric functionals.  To motivate these functionals, we first introduce an energy functional modeled on \eqref{OvalsSpecEqn}.  
Specifically, for a 
smooth curve, $\sigma$, of length $L(\sigma)$ and function, $f\in C^\infty(\sigma)$, set
\begin{equation}
 \label{MainIneq}
 \mathcal{E}_S[\sigma,f]= \int_{\sigma} |\nabla_\sigma f|^2 +\kappa^2_\sigma f^2 -\frac{(2\pi)^2}{L(\sigma)^2} f^2\; \mathrm{d}s.
\end{equation}
Clearly, the conjecture of Benguria and Loss is equivalent to the non-negativity of this functional.  For any strictly convex smooth curve, 
$\sigma$, set
\begin{equation} \label{OneIneq}
 \mathcal{E}_G[\sigma]=\int_{\sigma} \frac{|\nabla_\sigma \kappa_\sigma|^2}{4 \kappa^3_\sigma} - \frac{(2\pi)^2}{L(\sigma)^2}
\frac{1}{ \kappa_\sigma} \; \mathrm{d}s + {2\pi},  
\end{equation}
and 
\begin{equation} \label{TwoIneq}
 \mathcal{E}_G^*[\sigma]=\int_{\sigma} \frac{|\nabla_\sigma \kappa_\sigma|^2}{4 \kappa^2_\sigma}-\kappa^2_\sigma \; \mathrm{d}s + 
\frac{(2\pi)^2}{L(\sigma)}.
\end{equation}
Notice $\mathcal{E}_G$ is scale invariant, while $\mathcal{E}_G^*$ scales inversely with length.
 We will show that $\mathcal{E}_G$ and 
$\mathcal{E}_G^*$ are dual to each other in a certain sense -- justifying the notation.

Our main result is that the functionals \eqref{OneIneq} and \eqref{TwoIneq} are always non-negative and are zero only 
for  ovals.
\begin{thm}\label{MainThm}
If $\sigma$ is a smooth strictly convex closed curve  in $\Real^2$, then both $\mathcal{E}_G[\sigma]\geq 0$ and $\mathcal{E}_G^*[\sigma]\geq 
0$
with equality if and only if $\sigma\in \mathcal{O}.$ 
\end{thm}
To the best of our knowledge both inequalities are new.
Clearly, 
\[\mathcal{E}_G[\sigma]=\mathcal{E}_S[\sigma, \kappa^{-1/2}_\sigma],
\]
 and so the non-negativity of \eqref{OneIneq} would follow
from the
non-negativity of \eqref{MainIneq}.  Hence, Theorem \ref{MainThm} provides evidence for the
conjecture of Benguria and Loss.

We also discuss the natural symmetry of these functionals. To do so we need appropriate domains for the functionals.  To that end, we say a 
(possibly open) smooth planar curve is \emph{degree-one} if its unit tangent map is a degree one map from $\mathbb{S}^1$ to 
$\mathbb{S}^1$ 
-- for instance, any closed convex curve. A degree-one curve is \emph{strictly convex} if the unit tangent map is a diffeomorphism.    We 
show {(see \S3.3) that} there are natural (left and right) actions of 
$\mathrm{SL}(2,\Real)$ on $\mathcal{D}^\infty$, the space of {smooth,  degree-one curves} and on $\mathcal{D}_+^\infty$, the space of smooth 
strictly convex degree-one curves, which preserve the functionals. 
\begin{thm}\label{MainSymThm}
There are actions of $\mathrm{SL}(2,\Real)$ on $\mathcal{D}^\infty\times C^\infty$, the domain of $\mathcal{E}_S$, and on $\mathcal{D}^\infty_+$ the 
domain of $\mathcal{E}_G$ and $\mathcal{E}_G^*$ 
so that for $L\in \mathrm{SL}(2,\Real)$
\begin{equation*}
\begin{array}{cccc}
 \mathcal{E}_S[ (\sigma, f)\cdot L]=\mathcal{E}_S[\sigma,f], &
 \mathcal{E}_G[ \sigma\cdot L]=\mathcal{E}_G[\sigma],
& \mbox{and} &
 \mathcal{E}_G^*[ L\cdot \sigma]=\mathcal{E}_G^*[\sigma].
 \end{array}
\end{equation*}
Furthermore, there is an involution 
$*:\mathcal{D}_+^\infty\to \mathcal{D}_+^\infty$ so that 
\begin{equation*}
\begin{array}{ccc} *(L \cdot \sigma)= *\sigma\cdot L^{-1} & \mbox{and} & \mathcal{E}_G[*\sigma]=\frac{L(\sigma)}{2\pi} 
\mathcal{E}_G^*[\sigma]. 
\end{array}
\end{equation*}
\end{thm}
We observe that $\mathcal{O}$ is precisely the orbit of the round circle under these actions. Generically, the action does not 
preserve the condition of being a closed curve. Indeed, the image of the set of closed curves 
under this action is an open set in the space of curves and so is not well suited for the direct method in the calculus of variations.  
Arguably, this is the source of the difficulty in answering Benguria and Loss's conjecture.  Indeed, we prove Theorem~\ref{MainThm} in part 
by overcoming 
it.

\section{Preliminaries}
Denote by $\mathbb{S}^1=\set{x_1^2+x_2^2=1}\subset \Real^2$ the unit circle in $\Real^2$.  Unless otherwise stated, we always assume 
that 
$\mathbb{S}^1$ 
inherits the standard orientation from $\Real^2$ and consider $\mathrm{d}\theta$ to be the associated volume form and $\partial_\theta$ 
the dual 
vector field.  Abusing notation slightly, let $\theta:\mathbb{S}^1\to [0,2\pi )$ be the compatible chart with 
$\theta(\eE_1)=0$.  Let $\pi: \Real\to \mathbb{S}^1$ be the covering map so that $\pi^*\mathrm{d}\theta=\mathrm{d}x$ and $\pi(0)=\eE_1$ -- 
here $x$ is the 
usual coordinate on $\Real$.  Denote by $I: \mathbb{S}^1\to \mathbb{S}^1$ the involution given by $I(p)=-p$.  Hence, 
$\theta(I(p))=\theta(p)+\pi \mod 2\pi$. 

\begin{defn}\label{degree-onedef} An immersion $\sigma:[0, 
2\pi ]\to \Real^2$ is a \emph{degree-one curve of class $C^{k+1,\alpha}$}, if there is
\begin{itemize}
 \item a degree-one map $\TT_\sigma:\mathbb{S}^1\to \mathbb{S}^1$ of class $C^{k,\alpha}$, the \emph{unit tangent map of $\sigma$},
 \item a point $\xX_\sigma\in \Real^2$, the \emph{base point of $\sigma$}, and
 \item a value $L(\sigma)>0$, the \emph{length of $\sigma$}, 
\end{itemize}
so that
\begin{equation*}
 \sigma(t)=\xX_\sigma+\frac{L(\sigma)}{2\pi}\int_0^t \TT_\sigma(\pi(x)) \; \mathrm{d}x.
\end{equation*}
The curve $\sigma$ is \emph{strictly 
convex} provided the unit tangent map $\TT_\sigma$ has a $C^{k,\alpha}$ inverse and is \emph{closed} provided $\sigma(0)=\sigma(2\pi)$.
\end{defn}
A degree-one curve, $\sigma$, is uniquely determined by the data $(\TT_\sigma, \xX_\sigma, L(\sigma))$.
Denote by $\mathcal{D}^{k+1,\alpha}$ the set of degree-one curves of class 
$C^{k+1,\alpha}$ and by $\mathcal{D}^{k+1,\alpha}_+\subset \mathcal{D}^{k+1,\alpha}$ the set of strictly convex degree-one curves of of 
class 
$C^{k+1,\alpha}$.
The length element associated to $\sigma$ is $\mathrm{d}s=\frac{L(\sigma)}{2\pi} \mathrm{d}x=\frac{L(\sigma)}{2\pi}\pi^* \mathrm{d}\theta=\pi^*\widetilde{\mathrm{d}s}$.
If $\sigma\in \mathcal{D}^2$, then the geodesic curvature, $\kappa_\sigma$, of $\sigma$ satisfies 
$\kappa_\sigma=\pi^*\tilde{\kappa}_\sigma$ where $\tilde{\kappa}_\sigma\in C^{k-1,\alpha}(\mathbb{S}^1)$ satisfies 
$$\int_{\mathbb{S}^1} \tilde{\kappa}_\sigma \; \widetilde{\mathrm{d}s}= 
2\pi.$$
Conversely, given such a $\kappa_\sigma$ there is a degree-one curve with geodesic curvature $\kappa_\sigma$.
Abusing notation slightly, we will not distinguish between $\mathrm{d}s$ and $\widetilde{\mathrm{d}s}$ and between $\kappa_\sigma$ and 
$\tilde{\kappa}_\sigma$. Clearly, $\sigma \in\mathcal{D}^2_+$ if and only if 
$\kappa_\sigma>0$.  

The \emph{standard parameterization of $\mathbb{S}^1$} is given by the data 
$(\TT_0, \eE_1, 2\pi)$ where
\[ 
 \TT_0(p)=-\sin(\theta(p))\eE_1+\cos(\theta(p))\eE_2.
\]
Let $\Diff_+^{k,\alpha}(\mathbb{S}^1)$ denote the orientation preserving diffeomorphisms of $\mathbb{S}^1$ of class $C^{k,\alpha}$ -- 
that 
is bijective maps of class $C^{k,\alpha}$ with inverse of class $C^{k,\alpha}$.  Endow this space with the usual $C^{k,\alpha}$ 
topology. For $\sigma\in \mathcal{D}_+^{k+1,\alpha}$, we call the map 
$$
\phi_{\sigma}=\TT_{0}^{-1}\circ \TT_{\sigma}
$$ the \emph{induced diffeomorphism of $\sigma$}. Clearly, the induced diffeomorphism of the standard parameterization of $\mathbb{S}^1$
is the identity map. 

For $f \in C^k(\SS)$, let $f'=\partial_{\theta}f$, $f^{\prime\prime}=(f^{\prime})^{\prime}$ and likewise for higher order 
derivatives.  Observe that, for $\phi \in \mathrm{Diff}^1_+(\SS)$, we have $\phi'>0$ where $\phi'\in C^0(\mathbb{S}^1)$ satisfies 
$\phi^*\mathrm{d} \theta=\phi'\mathrm{d}\theta$.  A simple computation shows that if $\sigma\in \mathcal{D}_+^2$, then 
$\phi'_\sigma= {\kappa_\sigma} \frac{L(\sigma)}{2\pi} $.

\section{Symmetries of the functionals} \label{ProjGeomApp}

Consider the group homomorphism $\Gamma: \mathrm{SL}(2,\Real)\to \Diff^\infty_+(\mathbb{S}^1)$ given by 
\begin{equation*}
 \Gamma(L) = \xX\mapsto \frac{L\cdot \xX}{|L\cdot \xX|}
\end{equation*}
where $\xX\in \mathbb{S}^1$ and $L\in  \mathrm{SL}(2,\Real)$.~
Denote the image of $\Gamma$ by $\Mob(\mathbb{S}^1)$ which we refer to as the M\"{o}bius group of $\mathbb{S}^1$.  
One computes that
\begin{equation*}
 \TT_0\circ \Gamma(L)=\frac{ L \cdot \TT_0}{|L \cdot \TT_0|}.
\end{equation*}
These are precisely the unit tangent maps of the ovals of \cite{MR2091490}.
That is,
\begin{equation*}
 \mathcal{O}=\set{\sigma\in \mathcal{D}_+^\infty: \phi_\sigma\in \Mob(\mathbb{S}^1)}.
\end{equation*}

\subsection{The Schwarzian derivative}
For $\phi\in \Diff_+^3(\mathbb{S}^1)$ the \textit{Schwarzian derivative} of $\phi$ is defined as 
\begin{equation*}
S_{\theta}(\phi)=\frac{\phi^{\prime\prime\prime}}{\phi^{\prime}}-\frac{3}{2}\left(\frac{\phi^{\prime\prime}}{\phi^{\prime}}\right)^2,
\end{equation*}
where primes denote derivatives with respect to $\theta$. A fundamental feature of the Schwarzian derivative is that it satisfies the 
following co-cycle property 
\begin{equation} \label{CocycleCondBasic}
S_\theta(\phi\circ \psi)=\left(S_\theta(\phi)\circ \psi\right)\cdot (\psi^{\prime})^2+S_{\theta}(\psi),
\end{equation}
where $\phi,\psi\in \Diff_+^3(\mathbb{S}^1)$. 
After some computation, one verifies that the Schwarzian derivative gives the following intrinsic characterization of 
$\Mob(\mathbb{S}^1)$
\begin{equation}
 \label{SchwarzianMobS1eqn}
 \phi \in \Mob(\mathbb{S}^1) \iff S_\theta(\phi)+2(\phi')^2-2 =0.
\end{equation}

The Schwarzian derivative arises most naturally in the context of projective differential geometry.  This perspective also gives a 
conceptual proof of \eqref{SchwarzianMobS1eqn}. For this proof as well as the necessary background the reader may consult the 
Appendix~\ref{AppB} as well as the references cited there.

\subsection{Projective Symmetries}
We now describe the natural symmetries of \eqref{MainIneq}, \eqref{OneIneq} and \eqref{TwoIneq}.  We also introduce a notion of duality for 
strictly convex degree-one curves -- this duality will streamline some of the arguments. 
For $\sigma\in \mathcal{D}^{k+1,\alpha}$, define the \emph{dual curve},  $\sigma^*\in  \mathcal{D}^{k+1,\alpha}$
to be the unique curve with
\begin{equation*}\begin{array}{cccc}
 \phi_{\sigma^*}=\phi_{\sigma}^{-1}, &   \xX_{\sigma^*}= 
\xX_{\sigma} & \mbox{and} & L(\sigma^*)={L(\sigma)}.\end{array}
\end{equation*}
That is, $\sigma^*$ is the curve whose induced diffeomorphism is the inverse to the induced diffeomorphism of $\sigma$.  Clearly,  
$(\sigma^*)^*=\sigma$.
To proceed further, we note that the functionals $\eqref{OneIneq}$ and $\eqref{TwoIneq}$ can, by integrating by parts, be made to 
naturally involve the Schwarzian derivative. 
To see this fix  $\sigma \in \mathcal{D}_+^{4}$ with
$L(\sigma)=2\pi$.  As $\kappa_\sigma=\phi_\sigma'$,  
\begin{equation} \label{OneIneqSchwarzian}
 \begin{aligned}
\mathcal{E}_G[\sigma] &=\int_{\mathbb{S}^1} \frac{(\phi_\sigma'')^2}{4(\phi_\sigma')^3} -\frac{1}{\phi_\sigma'}+\phi_\sigma' \; \mathrm{d}\theta \\
 &=\int_{\mathbb{S}^1} \left( \frac{\phi_\sigma''}{2(\phi_\sigma')^3}\right)' \phi_\sigma'+\frac{3(\phi_\sigma'')^2}{4(\phi_\sigma')^3} 
-\frac{1}{\phi_\sigma'}+\phi_\sigma' \; \mathrm{d}\theta \\
&=\frac{1}{2} \int_{\mathbb{S}^1} \frac{S_{\theta}(\phi_\sigma) +2(\phi_\sigma')^2-2}{\phi_\sigma'} \; \mathrm{d}\theta,
\end{aligned}
\end{equation}
where the second equality follows by integrating by parts.  
Likewise,
\begin{equation} \label{TwoIneqSchwarzian}
 \mathcal{E}_G^*[\sigma] =-\frac{1}{2} \int_{\mathbb{S}^1} S_{\theta}(\phi_\sigma) +2(\phi_\sigma')^2-2 \; \mathrm{d}\theta.
\end{equation}
An immediate consequence of this is the following useful fact,
\begin{prop} \label{DualGProp}
 For $\sigma \in \mathcal{D}_+^{4}$, we have $\mathcal{E}_G[\sigma]=\frac{L(\sigma)}{2\pi}\mathcal{E}_G^*[\sigma^*]$.
\end{prop}
\begin{proof}
By scaling we may assume that $L(\sigma)=2\pi$. Write $\psi_\sigma=\phi_\sigma^{-1}$.
The co-cycle property for the Schwarzian derivative implies that
\begin{equation*}
 S_\theta(\psi_\sigma)=-\frac{S_\theta(\phi_\sigma)\circ \phi_\sigma^{-1}}{(\phi'_\sigma\circ \phi^{-1}_\sigma)^2},
\end{equation*}
where we have used that 
\[
\phi_\sigma'=\frac{1}{\psi_\sigma'\circ \phi_\sigma}. 
\]
Hence, by \eqref{OneIneqSchwarzian} and \eqref{TwoIneqSchwarzian}
\begin{align*}
 \mathcal{E}_G[\sigma]&=\frac{1}{2} \int_{\mathbb{S}^1} \frac{S_{\theta}(\phi_\sigma) +2(\phi_\sigma')^2-2}{\phi_\sigma'} \; \mathrm{d}\theta\\
  &= -\frac{1}{2}\int_{\mathbb{S}^1}( S_\theta(\psi_\sigma) 
\circ \phi_\sigma) \phi'_\sigma+2 (\psi'_\sigma\circ 
\phi_\sigma)^2\phi'_\sigma-2   \phi'_\sigma \;\mathrm{d}\theta\\
 &=-\frac{1}{2} \int_{\mathbb{S}^1} S_{\theta}(\psi_\sigma) +2(\psi_\sigma')^2-2 \; \mathrm{d}\theta\\
 &=\mathcal{E}_G^*[\sigma^*].
\end{align*}
\end{proof}

We now may define the desired actions. Consider first the  right action of $\Mob(\mathbb{S}^1)$ on $\mathcal{D}^{k+1,\alpha}$, 
$ \sigma\cdot \varphi = \sigma'$ where $\sigma'\in \mathcal{D}^{k+1,\alpha}$ is the unique element with
\begin{equation*}\begin{array}{cccc}
 \TT_{\sigma'}=\TT_{\sigma}\circ \varphi, &   \xX_{\sigma'}= 
\xX_{\sigma} & \mbox{and} & L(\sigma')=L(\sigma).\end{array}
\end{equation*}
  Notice, that if $\sigma\in \mathcal{D}^{k+1,\alpha}_+$ is strictly convex, then so is $\sigma'$ and in this case 
we have that $\phi_{\sigma'}=\phi_{\sigma}\circ \varphi$.
With this in mind, we also consider a left action of $\Mob(\mathbb{S}^1)$ on $\mathcal{D}^{k+1,\alpha}_+$, $\varphi\cdot \sigma =\sigma'$,
where $\sigma'\in \mathcal{D}^{k+1,\alpha}$ is the unique element with
\begin{equation*}\begin{array}{cccc}
\phi_{\sigma'}=\varphi \circ \phi_{\sigma}, &   \xX_{\sigma'}= 
\xX_{\sigma} & \mbox{and} & L(\sigma')=L(\sigma).\end{array}
\end{equation*}
We observe that for $\sigma \in \mathcal{D}^{k+1}_+$ and $\varphi\in \Mob(\mathbb{S}^1)$,
\[
 \varphi \cdot \sigma^*= (\sigma \cdot \varphi^{-1})^*.
\]
Finally, we define a right action of $\Mob(\mathbb{S}^1)$ on $C^{k,\alpha}(\mathbb{S}^1)$
by
\begin{equation*}
 f\cdot \varphi=\left(\varphi'\right)^{-1/2} f\circ\varphi .
\end{equation*}
If we use $\mathrm{d}\theta$ to identify $C^\infty(\mathbb{S}^1)$ with $\Omega^{-1/2}(\mathbb{S}^1)$, then this 
is the natural pull-back action on $\Omega^{-1/2}(\mathbb{S}^1)$ -- the space of weight $-1/2$ 
densities (see Appendix~\ref{AppB}).
\begin{prop}\label{SymProp}
 For any $\varphi\in \Mob(\mathbb{S}^1)$, $\sigma\in \mathcal{D}^\infty$ and $f\in C^\infty(\mathbb{S}^1)$,
 \begin{equation*}
   \mathcal{E}_S[\sigma, f]=\mathcal{E}_S[ \sigma\cdot \varphi,  f\cdot \varphi].
 \end{equation*}
 Likewise, for any $\varphi\in \Mob(\mathbb{S}^1)$ and $\sigma \in \mathcal{D}^\infty_+$,
\begin{equation*}  \begin{array}{ccc}
  \mathcal{E}_G[\sigma]=\mathcal{E}_G[ \sigma\cdot \varphi] & \mbox{and} &
  \mathcal{E}_G^*[\sigma]=\mathcal{E}_G^*[\varphi\cdot \sigma].
 \end{array}
  \end{equation*}
\end{prop}
\begin{proof}
By scaling, it suffices to take $L(\sigma)=2\pi$ so $\mathrm{d}s=\mathrm{d}\theta$.
Set
 \begin{equation*}
   f_{\varphi}'=(\varphi')^{1/2} f'\circ \varphi -\frac{1}{2} \frac{\varphi''}{(\varphi')^{3/2}} f\circ\varphi. 
 \end{equation*}
We show the first symmetry by computing, 
 \begin{align*}
   (f_{\varphi}')^2 &=(f'\circ\varphi)^2 \varphi'- \frac{\varphi''}{\varphi'} (f'\circ\varphi)( f\circ\varphi) +\frac{1}{4} 
\frac{(\varphi'')^2}{(\varphi')^{3}}(f\circ\varphi )^2\\
   &=(f'\circ\varphi)^2 \varphi'- \frac{1}{2} \frac{\varphi''}{(\varphi')^{2}}\partial_\theta \left( f\circ\varphi \right)^2 + 
\frac{1}{4} 
\frac{(\varphi'')^2}{(\varphi')^{2}}f_{\varphi}^2\\
   &=(f'\circ\varphi)^2 \varphi'-  \partial_\theta \left( \frac{\varphi'' (f\circ\varphi)^2 }{2 (\varphi')^{2}}  
\right)+ \left( \frac{\varphi''}{2(\varphi')^{2}}\right)' (f\circ\varphi)^2+ \frac{1}{4} 
\frac{(\varphi'')^2}{(\varphi')^{2}}f_{\varphi}^2\\   
   &=(f'\circ\varphi)^2 \varphi'- \partial_\theta \left( \frac{\varphi'' (f\circ\varphi)^2 }{2 (\varphi')^{2}}  
\right) + 
\frac{1}{2} 
S_\theta(\varphi) f_{\varphi}^2 \\  
   &=(f'\circ\varphi)^2 \varphi'- \partial_\theta \left( \frac{\varphi'' (f\circ\varphi)^2 }{2 (\varphi')^{2}}  
\right) + 
\left( 
1-(\varphi')^2\right)f_{\varphi}^2 +\frac{1}{2}  
(S_\theta(\varphi)+2(\varphi')^2-2)f_{\varphi}^2 \\
  &=(f'\circ\varphi)^2 \varphi'- \partial_\theta \left( \frac{\varphi'' (f\circ\varphi)^2 }{2 (\varphi')^{2}}  
\right) + 
 \left( 
1-(\varphi')^2\right)f_{\varphi}^2.  
 \end{align*}
 The last equality used $\varphi\in \Mob(\mathbb{S}^1)$ and \eqref{SchwarzianMobS1eqn}.
Integrating by parts gives,
\begin{align*}
  \int_{\mathbb{S}^1}(f_{\varphi}')^2 -f_{\varphi}^2\; \mathrm{d}\theta =\int_{\mathbb{S}^1}(f'\circ \varphi)^2 \varphi'  
-(f\circ\varphi )^2\varphi' \; \mathrm{d}\theta.
\end{align*}
Hence, after a change of variables
\begin{align*}
  \int_{\mathbb{S}^1}(f_{\varphi}')^2  -f_{\varphi}^2 \;\mathrm{d}\theta =\int_{\mathbb{S}^1}(f')^2  -f^2  \;\mathrm{d}\theta.
\end{align*}
Finally,
\begin{equation*}
 (\kappa_{\varphi} f_{\varphi})^2= \kappa_{\varphi}(\varphi(\theta))^2 f\circ\varphi ^2 \varphi'
\end{equation*}
and so a change of variables gives, 
\begin{equation*}
 \int_{\mathbb{S}^1} (\kappa_{\varphi} f_{\varphi})^2 \;\mathrm{d}\theta = \int_{\mathbb{S}^1} \kappa^2 f^2  \;\mathrm{d}\theta.
\end{equation*}
That is, $ \mathcal{E}_S[\sigma, f]=\mathcal{E}_S[ \sigma\cdot \varphi,  f\cdot \varphi]$.

The co-cycle property of the Schwarzian and \eqref{TwoIneqSchwarzian} immediately implies
\begin{align*}
 \mathcal{E}_G^*[\varphi \cdot \sigma]&=-\frac{1}{2}\int_{\mathbb{S}^1} S_\theta(\varphi\circ \phi_\sigma)+2 ((\varphi\circ\phi_\sigma)' 
)^2-2 \; \mathrm{d}\theta\\
  &=-\frac{1}{2}\int_{\mathbb{S}^1} S_\theta( \phi_\sigma)-2 (\phi_\sigma')^2 (\varphi'\circ \phi_\sigma)^2+2(\phi_\sigma')^2+2 
(\phi_\sigma')^2(\varphi'\circ\phi_\sigma )^2-2 \; \mathrm{d}\theta\\
  &=\mathcal{E}_G^*[\sigma]
\end{align*}
Finally, using Proposition \ref{DualGProp}
\begin{align*}
 \mathcal{E}_G[\sigma\cdot \varphi] &=\mathcal{E}_G^*[(\sigma\cdot \varphi)^*]= \mathcal{E}_G^*[\varphi^{-1} \cdot 
\sigma^*]=\mathcal{E}_G^*[\sigma^*]=\mathcal{E}_G[\sigma].
\end{align*}
\end{proof}
Theorem \ref{MainSymThm} is an immediate consequence of Propositions \ref{DualGProp} and \ref{SymProp} and the fact that 
$\Mob(\mathbb{S}^1)$ is isomorphic to 
$\mathrm{SL}(2,\Real)$.

As a final remark, we observe that we may extend the duality operator to $\mathcal{D}_+^{\infty} \times C^\infty(\mathbb{S}^1)$ and define a 
natural dual functional to $\mathcal{E}_S$.  Namely, set 
\begin{equation*}
 \begin{array}{ccc}(\sigma, f)^*=(\sigma^*, f\circ \phi_\sigma^{-1}) & \mbox{and} & \mathcal{E}_S^*[\sigma, f]=\int_{\sigma} \frac{|\nabla_\sigma f|^2}{\kappa_\sigma} -\kappa_\sigma f^2 +\frac{(2\pi)^2}{L(\sigma)^2} 
\frac{f^2}{ \kappa_\sigma} \; \mathrm{d}s.
\end{array}
\end{equation*}
We then have,
\begin{prop}\label{DualSProp}
If $\sigma\in \mathcal{D}_+^{\infty}$ and $f\in C^\infty(\mathbb{S}^1)$, then
\begin{equation*}
 \mathcal{E}_S[(\sigma, f)^*]=\frac{L(\sigma)}{2\pi}\mathcal{E}_S^*[\sigma,f].
\end{equation*}
\end{prop}
\begin{proof}
By scaling, we may assume that $L(\sigma)=2\pi$.
Writing $\psi_\sigma=\phi_\sigma^{-1}$, we compute
\begin{align*}
 \mathcal{E}_S[(\sigma, f)^*] &=\int_{\mathbb{S}^1} ((f\circ \psi_\sigma)')^2 +(\psi_\sigma')^2 (f\circ \psi_\sigma)^2-(f\circ 
\psi_\sigma)^2 \; \mathrm{d}\theta \\
 &=\int_{\mathbb{S}^1} (\psi_\sigma')^2 (f'\circ \psi_\sigma)^2 +(\psi_\sigma')^2 (f\circ \psi_\sigma)^2-\frac{(f\circ 
\psi_\sigma)^2}{\psi_\sigma'} \psi_\sigma' \; \mathrm{d}\theta\\
&=\int_{\mathbb{S}^1} (\psi_\sigma'\circ \psi_\sigma^{-1}) (f')^2 +(\psi_\sigma'\circ \psi_\sigma^{-1}) f^2-\frac{f^2}{\psi_\sigma'\circ \psi_\sigma^{-1}} \; \mathrm{d}\theta\\
&=\int_{\mathbb{S}^1} \frac{(f')^2}{\phi_\sigma'} + \frac{f^2}{\phi_\sigma'}-\phi_\sigma' f^2 \; \mathrm{d}\theta \\
=\mathcal{E}_S^*[\sigma,f].
\end{align*}

\end{proof}

\section{Deriving the geometric estimates}
To prove Theorem \ref{MainThm} we will use the direct method in the calculus of variations on an appropriate subclass of the class of 
degree-one convex curves.  This subclass is larger than the class of closed curves. We first note that the conjecture of 
Benguria and Loss holds for symmetric curves.
\begin{prop} \label{SymCurvesOvals}
 For $\sigma\in \mathcal{D}^2$, if the induced diffeomorphism satisfies $\phi_\sigma \circ I=I \circ \phi_\sigma$, 
then $\mathcal{E}_S[\sigma,f]\geq 0 $
with equality if and only if $\sigma\in \mathcal{O}$ and $f=\kappa_\sigma^{-1/2}$ is the lowest eigenfunction of $L_\sigma$.
\end{prop}
\begin{proof}
By scaling we may assume $L(\sigma)=2\pi$.
The symmetry implies that $\kappa_\sigma\circ I=\kappa_\sigma$ and $\TT_\sigma\circ I=-\TT_\sigma$. Hence, $\mathcal{E}_S[\sigma, 
f]=\mathcal{E}_S[\sigma, f\circ I]$ and so, the variational characterization of the lowest eigenvalue implies that 
the lowest eigenfunction $f$ must satisfy $f\circ I=f$.  
As observed in \cite{MR2091490}, 
\begin{equation*}
 \mathcal{E}_S[\sigma,f]= \int_{\mathbb{S}^1} |\yY'|^2-|\yY|^2 \; \mathrm{d}\theta
\end{equation*}
where $\yY=f \TT_\sigma$.  Moreover, $\yY(p)=(a \cos \theta(p)+b \sin \theta(p),c \cos \theta(p)+d \sin \theta(p) )$ if and only 
if $\sigma\in \mathcal{O}$.  As $\yY\circ I=-\yY$,
\begin{equation*}
 \int_{\mathbb{S}^1} \yY \; \mathrm{d}\theta=0
\end{equation*}
and the proposition follows from the one-dimensional Poincar\'{e} inequality.
\end{proof}
\begin{cor}
 \label{SymCurvesOne}
 For $\sigma\in \mathcal{D}^3_+$, if the induced diffeomorphism satisfies $\phi_\sigma\circ I=I \circ 
\phi_\sigma$, 
then $\mathcal{E}_G[\sigma]\geq 0$
with equality if and only if $\sigma\in \mathcal{O}$.
\end{cor}
\begin{proof}
 Take $f=\kappa_\sigma^{-1/2}$ in \eqref{MainIneq} and use Proposition \ref{SymCurvesOvals}.
\end{proof}
Motivated by \cite{Linde}, we make the 
following definition which is a weakening of the preceding symmetry condition.  
\begin{defn}
 A point $p\in \mathbb{S}^1$ is a \emph{balance point} of $\phi\in \Diff_+^0$ if $\phi(I(p))=I(\phi(p))$. Denote the number 
(possibly infinite) of balance points of 
$\phi$ by $n_B(\phi)\in \mathbb{N}\cup \set{\infty}$.  
\end{defn}
Clearly, if $p$ is a balance point then so is $I(p)$ and so $n_B(\phi)$ is even or $\infty$.  Further, it follows from the 
intermediate value theorem that $n_B(\phi)\geq 2$.  

Our definition of balance point is a slight generalization of Linde's \cite{Linde} 
notion of \emph{critical point} for convex closed curves.  Indeed, a critical point of a closed convex curve is just a balance point of its 
induced diffeomorphism. The key observation of Linde \cite[Lemma 2.1]{Linde} is that closed convex curves have 
at least six critical points. We will only need to know that there are at least four critical points and, so, for the sake of completeness, 
include an adaptation of Linde's argument to show this.
\begin{lem} \label{LindeLem}
If $\psi\in \Diff_+^1(\mathbb{S}^1)$ satisfies
\begin{equation*}
 \int_{\mathbb{S}^1}\psi' \cos \theta  \; \mathrm{d}\theta = \int_{\mathbb{S}^1}  \psi' \sin \theta \;\mathrm{d}\theta=0,
\end{equation*}
then $n_B(\psi)\geq 4$.  Hence, if $\sigma\in \mathcal{D}^2_+$ is closed, then $n_B(\phi_\sigma)\geq 4$.
\end{lem}
\begin{proof}
 As $\int_{\mathbb{S}^1} \psi' \; \mathrm{d}\theta=2\pi$ and $\psi'$ is continuous, there is a point $p_0$, so that if $\gamma_\pm$ are the 
components 
of $\mathbb{S}^1\backslash \set{p_0, I(p_0)}$, then $\int_{\gamma_\pm} \psi' \;\mathrm{d}\theta=\pi$.  That is, $p_0$ and $I(p_0)$ are balance points.  
Expanding $\psi'$ as a Fourier series, rotating so $\theta(p_0)=0$ and abusing notation slightly, gives that 
$$\psi=C+\theta+\sum_{n=2}^\infty \left(  a_n \sin n\theta + b_n \cos 
n\theta \right)=C+\theta+f(\theta)+g(\theta)$$ 
where $C$ is a constant,  $f$ are the remaining odd terms in the expansion and $g$ are the remaining even terms.
 By construction, $\psi(0)+\pi=\psi(\pi)$, $f(\theta+\pi)=-f(\theta)$ and $g(\theta+\pi)=g(\theta)$ and so $f(0)=0=f(\pi)$ and all balance points of $\psi$ in $\gamma_+$ correspond 
to zeros of $f$ in $(0,\pi)$.  If $f$ does not change sign on $(0,\pi)$, then either $f\equiv 0$ and $\psi$ has an infinite number of balance 
points, or $\int_{0}^\pi f(\theta) \sin \theta \; \mathrm{d}\theta \neq 0$.  However, as $f(\theta+\pi)\sin (\theta+\pi)=f(\theta) \sin 
\theta$, this would imply $\int_{0}^{2\pi} f(\theta) \sin \theta \; \mathrm{d}\theta \neq 0$ which is impossible.  Hence, $f$ must change sign and 
so 
$f$ has at least one zero in $(0,\pi)$ which verifies the first claim.

To verify the second claim. We first scale so $L(\sigma)=2\pi$.
 If $\sigma$ is closed, then $\int_{\mathbb{S}^1} \TT_\sigma \; \mathrm{d}\theta=0$. That is, $\int_{\mathbb{S}^1} \TT_0 \circ \phi_\sigma 
\; 
\mathrm{d}\theta=0$.  Changing variables, gives $\int_{\mathbb{S}^1} (\phi_\sigma^{-1})' \TT_0 \; \mathrm{d}\theta$.  Hence, 
$n_B(\phi_\sigma^{-1})\geq 4$ and it is clear that $n_B(\phi_\sigma)\geq 4$ as well.
\end{proof}

The spaces on which the functionals \eqref{OneIneq} and \eqref{TwoIneq} have good lower bounds seem to be 
spaces of curves whose induced diffeomorphisms have non-trivial number of balance points.  Motivated by this, set
\begin{equation*}
 \BDiff_+^{k, \alpha}(\mathbb{S}^1, N)=\set{\phi \in \Diff_+^{k, \alpha}(\mathbb{S}^1): n_B(\phi)\geq N}.
\end{equation*}
Hence, $\BDiff_+^{k, \alpha}(\mathbb{S}^1, 2)=\Diff_+^{k, \alpha}(\mathbb{S}^1)$ and $\Mob(\mathbb{S}^1)\subset 
\BDiff_+^{\infty}(\mathbb{S}^1, 
N)$ for all $N$.
Let $\overline{\BDiff}_+^{k,\alpha}(\mathbb{S}^1, N)$ be the closure of $\BDiff_+^{k, \alpha}(\mathbb{S}^1, N)$ in $\Diff_+^{k, 
\alpha}(\mathbb{S}^1)$, $\mathring{\BDiff}_+^{k,\alpha}(\mathbb{S}^1,N)$ be the interior and $\partial{\BDiff}_+^{k,\alpha}(\mathbb{S}^1, 
N)$ be the topological boundary. 
The function $n_B$ is not continuous on these spaces.  For example, the family $\phi_\lambda\in \Diff_+^\infty(\mathbb{S}^1)$ given by
\begin{equation} \label{BadFamily}
 \theta(\phi_\lambda (p))=2\cot^{-1} \left(\lambda \cot\left( \frac{1}{2} \theta(p)\right) \right),
\end{equation}
 for $\lambda>0$ has $n_B(\phi_\lambda)=2$ for $\lambda\neq 1$ and $n_B(\phi_1)=\infty$ and $\phi_\lambda\to \phi_1$ in 
$\Diff_+^\infty(\mathbb{S}^1)$ as $\lambda\to 1$.  Likewise, the family $\psi_\tau\in \Diff_+^{1,1}(\mathbb{S}^1)$, for $\tau\in \Real$ given 
by 
\begin{equation} \label{BadFamily2}
 \theta(\psi_\tau (p))=\left\{\begin{array}{cc} \cot^{-1}\left( \tau+\cot(\theta(p)-\frac{\pi}{2})\right)+\frac{\pi}{2} & 
\theta(p)\in\left[\frac{\pi}{2},\frac{3\pi}{2}\right] \\
                                                 \theta(p) & \theta(p)\in\left[0,\frac{\pi}{2}\right)\cup \left(\frac{3\pi}{2},2\pi\right),
                              \end{array} \right.
\end{equation}
 has $n_B(\psi_\tau)=2$ for $\tau\neq 0$ and $n_B(\psi_0)=\infty$.  Moreover, setting
\begin{equation} \label{BadFamily3}
 \psi_\tau^\lambda= \phi_{\lambda}^{-1} \circ \psi_\tau \circ \phi_{\lambda}\in \Diff_+^{1,1}(\mathbb{S}^1)
\end{equation}
gives a family so that for $\lambda>1$, $n_B(\psi_{\tau}^\lambda)=\infty$ and $\psi_\tau^\lambda\to \psi_\tau$ in $ 
\Diff_+^{1,\alpha}(\mathbb{S}^1)$ as $\lambda\to 1$ for any $\alpha\in[0,1)$. Observe that $\psi_\tau$ is the extension by the identity of the restriction of an element 
of $\Mob(\mathbb{S}^1)$ to one component of $\mathbb{S}^1\backslash \set{\eE_2,-\eE_2}$ and there 
are no other elements of $\Mob(\mathbb{S}^1)$ for which such an extension exists as an element of $\Diff_+^1(\mathbb{S}^1)$.

The elements of \eqref{BadFamily} show that $\Mob(\mathbb{S}^1)\subset \partial \BDiff_+^1(\mathbb{S}^1,4)$, while the 
elements of \eqref{BadFamily2} show that $\partial \BDiff_+^1(\mathbb{S}^1,4)$ contains elements with $n_B=2$.
In order to proceed further, we must refine the notion of balance point.
If $\phi\in \Diff_+^1(\mathbb{S}^1)$, then a balance point $p$ of $\phi$ is  \emph{stable} if and only if $\phi'(p)\neq \phi'(I(p))$ and is 
\emph{unstable} if $\phi'(p)=\phi'(I(p))$.  Denote the number of stable balance points 
of $\phi$  by $n_{SB}(\phi)$. For instance, the $\psi_\tau$ of \eqref{BadFamily2} have $n_{SB}(\psi_\tau)=0$.
\begin{lem}\label{ContBalLem}
 If $\phi\in 
\Diff_+^1(\mathbb{S}^1)$, then for each $N\in \mathbb{N}$ there is a $C^1$ 
neighborhood, $V=V_N$,  of $\phi$ so that $\min\set{n_{SB}(\phi),N}\leq 
n_{SB}(\psi)$ for all 
$\psi \in V$.  Furthermore, if $\phi$ satisfies 
$n_B(\phi)=n_{SB}(\phi)<\infty$, then 
$n_B$ is constant in a $C^1$ neighborhood of $\phi$.
\end{lem}
\begin{proof}
Let $B$ be the set of balance points of $\phi$ and $S\subset B$ be the set of 
stable balance points.   It follows from the inverse 
function theorem that for each $p\in S$, there is an open interval, $I_p$, in $\mathbb{S}^1$ so that $B\cap I_p=\set{p}$.
 It is straightforward to show, after fixing smaller open intervals, 
$I_p'$, 
satisfying $p\subset 
I_p'$ and $\bar{I}_p'\subset I_p$, that there are  $C^1$ 
neighborhoods, $V_p$, of $\phi$ in $\Diff_+^1(\mathbb{S}^1)$ so that all 
$\psi \in V_p$ have  only one stable balance point in 
$I_p' $ and no unstable balance points.  

If $n_{SB}(\phi)> N$, then let $S_N\subset S$ be 
some choice of $N$ distinct points of $S$, otherwise, let $S_N=S$.
As $S_N$ is finite, $V_N=\cap_{p\in S_N} 
V_p$ is a an open $C^1$ neighborhood of 
$\phi$ in $\Diff_+^1(\mathbb{S}^1)$ so that for any $\psi\in V_N$, there are 
$\min\set{n_{SB}(\phi),N}$ stable balance points in $U'_N=\cup_{p\in S_N} 
I'_p$ 
 and no unstable balance points.  Hence, $n_{SB}(\psi)\geq 
\min\set{n_{SB}(\phi),N}$ which completes the proof of the first claim.
The second claim follows by taking $N= n_{SB}(\phi)<\infty$. As 
$n_{B}(\phi)=n_{SB}(\phi)$, there are no balance points in 
$\mathbb{S}^1\backslash U'_N$ and 
so small $C^0$ perturbations of $\phi$ also have no balance points in 
$\mathbb{S}^1\backslash U'_N$. In other words, by shrinking $V_N$ one 
can ensure that  $n_{B}(\psi)=n_{SB}(\psi)=n_{SB}(\phi)=n_B(\phi)$ for 
all $\psi\in V_N$.
\end{proof}
\begin{lem} \label{StabBalLem}
 If $k\geq 1$ and $\phi \in  \partial{\BDiff}_+^{k, \alpha}(\mathbb{S}^1, 4)$, then $\phi$ has at least one pair of unstable 
balance points.
\end{lem}
\begin{proof}
 If $\phi \in  \partial{\BDiff}_+^{k, \alpha}(\mathbb{S}^1, 4)$ for $k\geq 1$, then $\phi \in  \partial{\BDiff}_+^{1}(\mathbb{S}^1, 4)$.  
Hence, we can restrict attention to the $C^1$ setting.
If $n_B(\phi)=2$, then the two balance points must be unstable as otherwise Lemma \ref{ContBalLem} would imply that any 
$C^1$ perturbation of $\phi$ also has only two balance points -- that is $\phi \not\in 
\overline{\BDiff}_+^{1}(\mathbb{S}^1, 4)$.  If $n_B(\phi)=\infty$, then it must have some unstable balance 
points, as the set balance points is closed while the set of stable balance points is discrete and so is a proper subset.  Finally, if $4\leq n_B(\phi)<\infty$, then Lemma 
\ref{ContBalLem} implies at least two of them are unstable.  Otherwise, any $C^1$ perturbation of $\phi$ would 
continue to have at least four balance points, i.e., 
 $\phi$ is in the interior of ${\BDiff}_+^{1}(\mathbb{S}^1, 4)$.
\end{proof}

We next introduce the appropriate energy space for $\mathcal{E}_G^*$ -- we work with this functional as it has nicer analytic properties.  
It will be convenient to think of $\mathcal{E}_G^*$ as a functional on $\Diff_+^2(\mathbb{S}^1)$ by considering 
$\mathcal{E}_G^*[\phi]=\mathcal{E}_G^*[\sigma]$ where $\phi=\phi_\sigma$ and $L(\sigma)=2\pi$.
To motivate our choice of energy space set $u=\log \phi'\in C^\infty(\mathbb{S}^1)$.  
Notice, that $u$ satisfies the non-linear constraint
$$
\int_{\mathbb{S}^1} e^u \; \mathrm{d}\theta =2\pi.
$$
A simple change of variables shows that the functional
\begin{equation*}
 E[u]=\int_{\mathbb{S}^1}\frac{1}{4} (u')^2-e^{2u} \; \mathrm{d}\theta 
\end{equation*}
satisfies $E[u]+2\pi=\mathcal{E}_G^*[\phi]$.
The Euler-Lagrange equation of $E[u]$ with respect to the constraint is a semi-linear ODE of the form
\begin{equation}\label{SemiLinODE}
 \frac{1}{4}u''+ e^{2u}+\beta e^u =0
\end{equation}
for some $\beta$.
Define the following energy space for $E$
\[
H^1_{2\pi}(\mathbb{S}^1)=\set{u\in H^1(\mathbb{S}^1):\int_{\mathbb{S}^1} \mathrm{e}^u \; \mathrm{d}\theta=2\pi}\subset   H^1(\mathbb{S}^1).
\]
The Sobolev embedding theorem implies $H^1(\mathbb{S}^1)\subset C^{1/2}(\mathbb{S}^1)$.  Hence, 
$H^1_{2\pi}(\mathbb{S}^1)$ is a closed subset of $H^1(\mathbb{S}^1)$ with respect to the weak topology of  $H^1(\mathbb{S}^1)$.  
Let
\begin{equation*}
\HDiff_+(\mathbb{S}^1)=\set{\phi\in \Diff_+^{1}(\mathbb{S}^1): \log \phi' \in H^1_{2\pi}(\mathbb{S}^1)}\subset  
\Diff^{1,1/2}_+(\mathbb{S}^1).
\end{equation*}
have a strong (resp.~weak) topology determined by $\phi_i\to \phi$ when $\log \phi'_i \to \log \phi'$ strongly in $H^1(\mathbb{S}^1)$ 
 (resp.~weakly in  $H^1(\mathbb{S}^1)$). Clearly, $\mathcal{E}_G^*$ extends to $\HDiff_+(\mathbb{S}^1)$.  As $\Diff^{1,1}_+(\mathbb{S}^1)\subset \HDiff_+(\mathbb{S}^1)$, the family given by \eqref{BadFamily2} satisfies $\psi_\tau\in\HDiff_+(\mathbb{S}^1)$ and one computes that $\mathcal{E}_G^*[\psi_\tau]=0$.

We will need the following smoothing lemma:
\begin{lem}\label{SmoothingLem}
 For $\phi \in  \HDiff_+(\mathbb{S}^1)$, there exists a sequence $\phi_i\in \Diff_+^\infty(\mathbb{S}^1)$ with $\phi_i\to \phi$ in the 
strong topology of $\HDiff_+(\mathbb{S}^1)$.  Furthermore, if $\phi$ satisfies $\phi\circ I=I\circ \phi$, then the $\phi_i$ may be chosen 
so $\phi_i\circ I=I \circ \phi_i$.
\end{lem}
\begin{proof}
 Fix $p_0\in 
\mathbb{S}^1$, let $\nu_{\epsilon}(p, p_0)$ be a family of $C^\infty$ mollifiers with $\nu_{\epsilon}(p,p_0)\geq 0$, 
$\mathrm{supp}(\nu_{\epsilon}(\cdot, p_0))\subset B_{\epsilon}(p_0)$,
$\nu_{\epsilon}(p,p_0)=\nu_{\epsilon}(p_0,p)$, $\nu_\epsilon(I(p_0),I(p))=\nu_\epsilon(p_0,p)$ and $\int_{\mathbb{S}^1} \nu_{\epsilon}(p , 
p_0) \; \mathrm{d}\theta(p) =1$. That is, $\lim_{\epsilon\to 0} \nu_\epsilon (p, p_0)= 
\delta_{p_0}(p)$ the Dirac delta with mass at $p_0$. Set
\[
P_\epsilon=\int_{\mathbb{S}^1} \nu_{\epsilon} (\cdot , p ) 
 \phi'(p)\; \mathrm{d}\theta (p)\in C^\infty(\mathbb{S}^1). 
\]
Hence, $\int_{\mathbb{S}^1} P_\epsilon \; \mathrm{d}\theta=2\pi$ and $P_{\epsilon}\geq \min_{\mathbb{S}^1} \phi'>0$. It follows, that there are $\phi_\epsilon \in 
\Diff_+^\infty(\mathbb{S}^1)$ so that $\phi_\epsilon(p_0)=\phi(p_0)$ and $\phi_\epsilon'=P_\epsilon$. 
As  $\log \phi'\in H^1(\mathbb{S}^1)$, $ \phi'\in H^1(\mathbb{S}^1)$ and so $P_\epsilon \to\phi'$ strongly in $H^1(\mathbb{S}^1)$.  This convergence together with the uniform lower bound on $P_\epsilon$ and the Sobolev embedding theorem implies that   $\log P_\epsilon$ 
converge strongly in $H^1(\mathbb{S}^1)$ to $\log 
\phi'$ -- that is, $\phi_\epsilon\to\phi$ strongly in $\HDiff_+(\mathbb{S}^1)$.

Finally, we observe that if $\phi\circ I=I\circ 
\phi$, then $\phi'\circ 
I=\phi'$ and so $P_\epsilon \circ 
I=P_\epsilon$. In particular, if $\phi\circ I=I\circ 
\phi$, then $\phi_\epsilon\circ I=I\circ \phi_\epsilon$.  
\end{proof}
\begin{lem} \label{SymCurvesTwoLem}
 If $\phi\in \HDiff_+(\mathbb{S}^1)$ satisfies $\phi\circ I=I\circ \phi$, then $\mathcal{E}_G^*[\phi]\geq 0$
 with equality if and only if $\phi\in \Mob(\mathbb{S}^1)$.
\end{lem}
\begin{proof}
 By Lemma \ref{SmoothingLem}, there are a sequence of $\phi_i\in \Diff^\infty_+(\mathbb{S}^1)$, with $\phi_i\circ I=I\circ \phi_i$ and 
$\phi_i\to \phi$ strongly in $\HDiff_+(\mathbb{S}^1)$.  In particular, $\mathcal{E}_G^*[\phi_i]\to \mathcal{E}_G^*[\phi]$.  Set 
$\psi_i=\phi_i^{-1}$ and note that $\psi_i\circ I=I\circ \psi_i$.  Further, let $\sigma_i\in \mathcal{D}^+_i$ have induced diffeomorphism 
$\psi_i$. 
By \eqref{TwoIneqSchwarzian}, Proposition \ref{DualGProp} and Corollary \ref{SymCurvesOne}, 
\[
\mathcal{E}_G^*[\phi_i]=\mathcal{E}_G^*[\sigma^*_i]=\mathcal{E}_G[\sigma_i]\geq 0,
\]
proving the desired inequality.
If one has equality, then the inequality implies that $\phi$ is critical with respect to variations preserving the symmetry.  It follows 
that $\phi$ is smooth and so $\phi\in \Mob(\mathbb{S}^1)$ by Corollary \ref{SymCurvesOne}.
\end{proof}

A symmetrization argument and Lemma \ref{SymCurvesTwoLem} imply:
\begin{prop} \label{BoundaryLowBndProp} If $\phi\in \HDiff_+(\mathbb{S}^1)\cap 
\partial{\BDiff}_+^1(\mathbb{S}^1, 4)$, then 
 \begin{equation*}
  \mathcal{E}_G^*[\phi]\geq 0
 \end{equation*}
 with equality if and only if $\phi=\psi_\tau\circ \hat{\phi}$ where $\psi_\tau$ is of the form \eqref{BadFamily2} for some $\tau\in \Real$ 
and $\hat{\phi}\in \Mob(\mathbb{S}^1)$.
 \end{prop}
\begin{proof}
Let $\phi \in \HDiff^+(\mathbb{S}^1)\cap 
\partial{\BDiff}_+^1(\mathbb{S}^1, 4)$. As $\phi \in \Diff_+^{1,\frac{1}{2}}(\mathbb{S}^1)$, Lemma \ref{StabBalLem} implies that 
$\phi$ has at least one unstable balance point $p_0$. Let $\gamma_\pm$ be the two components of $\mathbb{S}^1\backslash \set{p_0, 
I(p_0)}$. 
Up to relabeling, we may assume that
\begin{equation*}
 \mathcal{E}_G^*[\phi]\geq 2 \left( \int_{\gamma_-} \frac{1}{4}(u')^2 -\mathrm{e}^{2u} \; \mathrm{d}\theta+2\pi\right)
 \end{equation*}
 where $u=\log \phi'$.
Now define 
\begin{equation*}
\tilde{u}(p)=\left \{ \begin{array}{cc} u(p) & p\in \bar{\gamma}_- \\ u(I(p)) & p \in \gamma_+ \end{array} \right.
\end{equation*}
Here, $\bar{\gamma}_-$ is the closure of $\gamma_-$ in $\mathbb{S}^1$.  
Clearly, $\tilde{u}$ is continuous,  $\int_{\mathbb{S}^1} \mathrm{e}^{\tilde{u}} \; \mathrm{d}\theta=2\pi$ and
\begin{equation*}
 \mathcal{E}_G^*[\phi]\geq E[\tilde{u}] +2\pi.
 \end{equation*}
Hence, there is a $\tilde{\phi}\in \Diff_+^{1,1}(\mathbb{S}^1)\subset \HDiff_+(\mathbb{S}^1)$ so that $\tilde{u}=\log\tilde{\phi}'$.  By 
construction, $\tilde{\phi}\circ I =I\circ 
\tilde{\phi}$ and so by Lemma \ref{SymCurvesTwoLem}
\begin{equation*}
  \mathcal{E}_G^*[\phi]\geq \mathcal{E}_G^*[\tilde{\phi}]\geq 0, 
\end{equation*}
with equality if and only if $\tilde{\phi}\in \Mob(\mathbb{S}^1)$.  

In the case of equality for $\phi$
we could reflect either $\gamma_+$ or $\gamma_-$, hence the preceding argument implies, $\phi|_{\gamma_\pm}=\phi_\pm$ for $\phi_\pm \in 
\Mob(\mathbb{S}^1)$ which satisfy $\phi_\pm(\gamma_+)=\gamma_+$. By precomposing with a rotation, we may assume that 
$\set{p_0,I(p_0)}=\set{\eE_2,-\eE_2}$ and $\theta(\gamma_+)=\left(\frac{\pi}{2},\frac{3\pi}{2}\right)$.
Taking $\hat{\phi}=\phi_-\in \Mob(\mathbb{S}^1)$, one has $\phi\circ 
\hat{\phi}^{-1}\in \Diff^1_+(\mathbb{S}^1)$ and is the identity map on $\gamma_-$ and some element of $\Mob(\mathbb{S}^1)$ on $\gamma_+$.
This implies that $\phi\circ\hat{\phi}^{-1}=\psi_\tau$ where $\psi_\tau$ is of the form $\eqref{BadFamily2}$ for some 
$\tau\in \Real$. That is, $\phi=\psi_\tau \circ \hat{\phi}$.
\end{proof}

We next analyze certain ODEs generalizing \eqref{SemiLinODE}.  
\begin{prop}\label{ODEProp}
Fix $\gamma\geq 2\pi$. If $u\in C^\infty(\mathbb{S}^1)$ satisfies the ODE
 \begin{equation} \label{ODEeqn}
  \frac{1}{4} u'' - \alpha \mathrm{e}^{2u}+\beta \mathrm{e}^u =0
 \end{equation}
and the constraints
\begin{equation} \label{ODEconstraintsEqn}
 \int_{\mathbb{S}^1} \mathrm{e}^u \; \mathrm{d}\theta=2\pi \mbox{ and }  \int_{\mathbb{S}^1} \mathrm{e}^{2u} \; \mathrm{d}\theta=\gamma,
\end{equation}
then either $\gamma=2\pi$, $\alpha=\beta$ and $u \equiv 0$ or $\gamma>2\pi$ and there is an $n\in \mathbb{N}$ so that 
$\alpha=-n^2$ and $\beta=-\frac{\gamma}{2\pi} n^2$ and
\begin{equation*}
 u(p)=-\log\left(\frac{\gamma}{2\pi}+\sqrt{\left(\frac{\gamma}{2\pi}\right)^2-1} \cos( n(\theta(p)-\theta_0))\right)
\end{equation*}
for some $\theta_0$.
In this case,
\begin{equation*}
 E[u]=-2\pi \frac{n^2}{4} +\frac{(n^2-4) }{4}\gamma.
\end{equation*}
Hence, if $n\geq 2$, then 
\begin{equation*}
 E[u]\geq -2\pi.
\end{equation*}
with equality if and only if $\gamma=2\pi$ or $n=2$.
\end{prop}
\begin{proof}
It is straightforward to see that \eqref{ODEeqn} has the conservation law
\begin{equation*}
 \frac{1}{4}(u')^2-\alpha \mathrm{e}^{2u}+ 2 \beta \mathrm{e}^u = \eta.
\end{equation*}
Integrating this we see that
\begin{equation*}
E[u]+(1-\alpha) \gamma + 4\pi \beta= 2\pi \eta.
\end{equation*}
However, integrating \eqref{ODEeqn} gives that
\begin{equation*}
 -\alpha \gamma + 2\pi \beta=0
\end{equation*}
and hence
\begin{equation*}
 E[u]=2\pi \eta-\gamma- 2\pi \beta.
\end{equation*}

Now set $U=\mathrm{e}^{-u}$ one has that
\begin{equation*}
 \frac{1}{4}U'' =-\frac{1}{4}\mathrm{e}^{-u} u'' +\frac{1}{4}\mathrm{e}^{-u} (u')^2=-\alpha \mathrm{e}^{u}+\beta +\alpha \mathrm{e}^u -2\beta +\eta \mathrm{e}^{-u}=\eta U 
-\beta
\end{equation*}
That is, $U$ satisfies
\begin{equation*}
 U''-4\eta U= -4\beta.
\end{equation*}
As $U\in C^\infty(\mathbb{S}^1)$, either $U= \frac{\beta}{\eta}$, or $4\eta=-n^2$ for some $n\in \mathbb{Z}^+$ and 
\[
 U= \frac{\beta}{\eta}+C_1 \cos \sqrt{-4\eta} \theta + C_2 \sin \sqrt{-4\eta} \theta
\]
for some constants $C_1, C_2$. 
In the first case, the constraints force $\eta= \beta$ and so $u=0$, $\alpha=\beta=\eta$, $\gamma=2\pi$ and $E=-2\pi$. 

In the second case, we first note that $U>0$ and so 
\begin{equation*}
 \frac{\beta}{\eta}> \sqrt{C_1^2+C_2^2}.
\end{equation*}
Using the calculus of residues, we compute that
\begin{align*}
 \int_{\mathbb{S}^1} \mathrm{e}^u \; \mathrm{d}\theta & = \int_{\mathbb{S}^1} \frac{1}{U} \; \mathrm{d}\theta \\
                            &= \int_{\mathbb{S}^1} \frac{1}{\frac{\beta}{\eta} 
+\frac{C_1}{2}(z^n+z^{-n})+\frac{C_2}{2i}(z^n-z^{-n})} \; \frac{\mathrm{d}z}{i z}
\\&= \frac{2\pi}{\sqrt{\left(\frac{\beta}{\eta}\right)^2-C_1^2-C_2^2}}.
\end{align*}
Keeping in mind that $U>0$, the first constraint is satisfied if and only if
\[
\beta= \eta\sqrt{1+C_1^2+C_2^2}.
\]
Hence, 
\begin{equation*}
 u= -\log \left(\sqrt{1+C_1^2+C_2^2}+C_1 \cos \sqrt{-4\eta}\theta + C_2 \sin \sqrt{-4\eta} \theta\right).
\end{equation*}
Plugging this into \eqref{ODEeqn}, shows that  $\alpha=\eta$. Hence, 
\begin{equation*}
 \gamma=2\pi \beta/ \alpha=2\pi \sqrt{1+C_2^2+C_2^2}.
\end{equation*}
We conclude that, 
\begin{equation*}
 E[u]=2\pi \eta- \gamma -\eta \gamma=-2\pi \frac{n^2}{4} +(\frac{1}{4}n^2-1) \gamma.
\end{equation*}
Hence, if $n\geq 2$,  then as $\gamma\geq 2\pi$
\[
E[u]\geq -2\pi \frac{n^2}{4} +2\pi (\frac{n^2}{4}-1)\geq -2\pi.
\]
with equality if and only if $n=2$.
\end{proof}
\begin{rem}
 If $n=1$, then as $\gamma\to \infty$, $E[u]\to 0$.  If $n=2$, then $u=\log \phi'$ for $\phi\in 
\Mob(\mathbb{S}^1)$.
\end{rem}

Combining Propositions \ref{BoundaryLowBndProp} and \ref{ODEProp} gives:
\begin{prop} \label{MainEstProp}
 If $\phi\in \HDiff_+(\mathbb{S}^1)\cap 
\overline{\BDiff}_+^1(\mathbb{S}^1, 4)$, then 
 \begin{equation*}
  \mathcal{E}_G^*[\phi]\geq 0
 \end{equation*}
 with equality if and only if  $\phi=\psi_\tau\circ \hat{\phi}$ where $\psi_\tau$ is of the form \eqref{BadFamily2} for some $\tau\in \Real$ 
and $ \hat{\phi}\in \Mob(\mathbb{S}^1)$.  If, in addition, $\phi \in \Diff_+^2(\mathbb{S}^1)$ or $\phi\in 
\BDiff_+^1(\mathbb{S}^1,4)$, then equality occurs if and only if $\phi \in \Mob(\mathbb{S}^1)$.
\end{prop}
\begin{rem}
 This result is sharp in that the inequality fails for \eqref{BadFamily}.  
\end{rem}
\begin{proof}
 If inequality does not hold, then there is a $\phi_0\in  \HDiff_+(\mathbb{S}^1)\cap 
\overline{\BDiff}_+^1(\mathbb{S}^1, 4)$ so that $\mathcal{E}_G^*[\phi_0]<0$.  Let $u_0=\log \phi_0'$ and
 set $\gamma_0=\int_{\mathbb{S}^1} (\phi')^2 \; \mathrm{d}\theta=\int_{\mathbb{S}^1} \mathrm{e}^{2u_0} \; \mathrm{d}\theta$. The Cauchy-Schwarz inequality implies 
that $\gamma_0\geq 2\pi$ with equality if and 
only if $u_0\equiv 0$. Now consider the minimization problem
 \begin{equation}
  E(\gamma)= \inf\set{\mathcal{E}_G^*[\phi] \; \phi 
\in  \HDiff_+(\mathbb{S}^1)\cap \overline{\BDiff}_+^1(\mathbb{S}^1, 4), \int_{\mathbb{S}^1} 
(\phi')^2\; \mathrm{d}\theta=\gamma}.
 \end{equation} 
Clearly, our assumption ensures that $E(\gamma_0)\leq \mathcal{E}_{G}^*[\phi_0]<0$. Notice 
without the constraint $\int_{\mathbb{S}^1} 
(\phi')^2\; \mathrm{d}\theta$  the 
symmetry of Theorem \ref{MainSymThm}
would imply that  $E$ is not 
coercive for the $H^1$-norm of  $u=\log \phi'$.
  However, with the constraint we are minimizing the Dirichlet energy of $u$
and so the Rellich compactness theorem 
gives a 
$u_{\min}\in H^1_{2\pi}(\mathbb{S}^1)$ satisfying
\[
 E(\gamma_0)=\int_{\mathbb{S}^1}\frac{1}{4} (u'_{\min})^2 -\mathrm{e}^{2u_{\min}} \; \mathrm{d}\theta+2\pi=\int_{\mathbb{S}^1}\frac{1}{4} 
(u'_{\min})^2 \; 
\mathrm{d}\theta-\gamma_0+2\pi<0.
\]
and, hence, a $\phi_{\min}\in \HDiff_+(\mathbb{S}^1)\cap\overline{\BDiff}_+^1(\mathbb{S}^1, 4)$ so that $\log \phi'_{\min} =u_{\min}$.  
However, 
Proposition \ref{BoundaryLowBndProp} implies that $\phi_{\min} \in   \mathring{\BDiff}_+^1(\mathbb{S}^1, 4)$.
This implies that $u_{\min}$ is critical with respect to arbitrary variations in $H^1(\mathbb{S}^1)$ which preserve the 
constraints
\begin{equation*}
 \begin{array}{ccc} \int_{\mathbb{S}^1} \mathrm{e}^u \; \mathrm{d}\theta=2\pi& \mbox{and} & \int_{\mathbb{S}^1} \mathrm{e}^{2u} \; 
\mathrm{d}\theta=\gamma_0 .
\end{array}
\end{equation*}
Hence,  $u_{\min}$ weakly satisfies the 
Euler-Lagrange equation
\begin{equation*}
  \frac{1}{4}u''_{\min} -\alpha \mathrm{e}^{2 u_{\min}}+\beta \mathrm{e}^{u_{\min}}=0.
\end{equation*}
As this is a semi-linear ODE and $u_{\min}\in C^{1/2}(\mathbb{S}^1$) by Sobolev embedding, $u_{\min}\in C^{2+\alpha}(\mathbb{S}^1)$ and 
satisfies this equation classically. Hence, $u_{\min}$ is smooth by standard ODE theory.  
Notice, that if $\phi_{\lambda}$ is one of the elements of \eqref{BadFamily}, then 
$$u_{\lambda}=\log \phi_{\lambda}'=-\log\left( \frac{1}{2}(\lambda+\lambda^{-1})+\frac{1}{2}(\lambda-\lambda^{-1} )\cos \theta(p)\right).$$
Applying,  Proposition \ref{ODEProp} to $u_{\min}$ we see that, up to a rotation, if $n=1$, then $\phi_{\min}=\phi_{\lambda}$ for some 
$\lambda$.  As $n_{B}(\phi_{\min})\geq 4$, this is impossible.  Hence, $n\geq 2$ and so $E[u_{\min}]\geq 0$ which contradicts $E(\gamma_0)< 
0$ and proves the inequality.  

Equality cannot hold for $\phi\in \mathring{\BDiff}_+^1(\mathbb{S}^1,4)$.  If it did, $\phi$ would be a 
critical point for $\mathcal{E}_G^*$ with respect to arbitrary variations in $\HDiff_+(\mathbb{S}^1)$.  Applying Proposition \ref{ODEProp} 
to $u=\log \phi'$ shows this is impossible.  Hence, equality is only achieved on $\partial \BDiff_+^1(\mathbb{S}^1,4)$ and so the claim follows from Proposition \ref{BoundaryLowBndProp} and the observation that, for $\psi_\tau$ as in \eqref{BadFamily2}, $\psi_\tau \in 
\Diff_+^2(\mathbb{S}^1)$ or $\BDiff_+^1(\mathbb{S}^1, 4)$ if and only if $\tau=0$.
\end{proof}

We may now conclude the main geometric estimates.
\begin{proof}[Proof of Theorem \ref{MainThm}]
 The natural scaling of the problem means that we may apply a homothety to take $L(\sigma)=2\pi$. 
 As $\sigma$ is a smooth closed strictly convex curve, it is a smooth degree-one strictly convex curve.  Let $\phi_\sigma\in 
\Diff^+(\mathbb{S}^1)$, be the induced diffeomorphism and let $\psi_\sigma=\phi^{-1}_\sigma$. 
By Lemma \ref{LindeLem}, $\phi_\sigma, \psi_\sigma\in \BDiff_+^1(\mathbb{S}^1,4)$.  The claim now follows from Propositions
\ref{MainEstProp} and \ref{DualGProp}.
\end{proof}

\appendix
\section{On extending the conjecture of Benguria and Loss}
Benguria and Loss's conjecture concerns closed curves.  In light of the present paper, specifically
the symmetry of Theorem \ref{MainSymThm}, it 
is tempting to think that their conjecture can be extended to degree-one curves with more than two balance points.  However, this is not the 
case.
\begin{lem}
 For every $N\in \mathbb{N}$, there is a $\sigma\in \mathcal{D}^\infty_+$ so that $\phi_\sigma\in \BDiff_+^\infty(\mathbb{S}^1,N)$ and 
$\mathcal{E}_S[\sigma,f]<0$ for some function $f\in C^\infty(\mathbb{S}^1)$.
\end{lem}
\begin{proof}
 Consider $\sigma_\tau$ to be the curve in $\mathcal{D}_+^{2,1}$  which has $\sigma_\tau(\eE_1)=\eE_1$, 
$L(\sigma_\tau)=2\pi$ and induced diffeomorphism $\phi_{\sigma_\tau}=\psi_\tau$ where $\psi_\tau$ is given by \eqref{BadFamily2}.  Note, 
that for $\tau\neq 0$, $\sigma_\tau$ is not closed.  One 
computes that for $f_\tau=\kappa_{\sigma_\tau}^{-1/2}\in C^{0,1}(\mathbb{S}^1)\subset H^1(\mathbb{S}^1)$, that 
$\mathcal{E}_S[\sigma_\tau,f_\tau]=\mathcal{E}_G[\sigma_\tau]=0$.  However,  
$$
L_{\sigma_\tau}f_\tau=-f_\tau'' +\kappa_{\sigma_\tau}^2 f_\tau = f_\tau +C(\tau)\delta_{\eE_2}-C(\tau)\delta_{-\eE_2},
$$
 distributionally and the constant $C(\tau)\neq 0$ if and only if $\tau\neq 0$.  Hence, for 
$\tau\neq 0$, $f_\tau$ is not an eigenfunction and so there must be a $\hat{f}_\tau\in C^2(\mathbb{S}^1)$ with
$\mathcal{E}_S[\sigma_\tau,\hat{f}_\tau]<0$.  Consider the elements $\psi_{\tau}^\lambda\in \Diff_+^{1,1}(\mathbb{S}^1)$ given by 
\eqref{BadFamily3} and pick $\sigma_{\tau}^\lambda\in \mathcal{D}_+^{2,1}$ so that $\sigma_\tau^\lambda(\eE_1)=\eE_1$, 
$L(\sigma_\tau^\lambda)=2\pi$ and the induced diffeomorphism is
$\psi_{\tau}^\lambda$.  Clearly,  $\sigma_{\tau}^\lambda\to \sigma_\tau$ as $\lambda\to 1$ in the $C^2$ topology.  Hence, 
$\mathcal{E}_S[\sigma_{\tau}^{\lambda},\hat{f}_\tau]\to \mathcal{E}_S[\sigma_{\tau},\hat{f}_\tau]$ as $\lambda\to 1$.  Hence, for $\tau\neq 
0$ and $\lambda>1$ sufficiently close to $1$, we obtain a $\sigma\in \mathcal{D}_+^\infty$ with $n_B(\sigma)=\infty$ 
and $\mathcal{E}_S[\sigma, \hat{f}_\tau]<0$ by smoothing out $\sigma_{\tau}^\lambda$ as in Lemma \ref{SmoothingLem}. 
Smoothing out $\hat{f}_\tau$ gives $f$ so that $\mathcal{E}_S[\sigma,f]<0$.
\end{proof}

\section{Projective structures}\label{AppB}

We review some basic concepts from projective differential geometry which will motivate the definition of $\Mob(\mathbb{S}^1)$ made above 
as well as provide the natural context for the symmetries of the functionals of $\eqref{MainIneq}, \eqref{OneIneq}$ and 
$\eqref{TwoIneq}$. This is a vast subject with many different perspectives and we present only a summarized version. We refer the 
interested 
reader to the excellent book \cite{MR2177471} by  Ovsienko and Tabachnikov as well as their article \cite{MR2489717} -- these were our 
main 
sources for this material.
\subsection{One-Dimensional Projective Differential Geometry}
Let $M$ be a one-dimensional oriented manifold. {We fix a square root $(T^*M)^{1/2}$ of the cotangent bundle of $M$ so that we have an isomorphism of line bundles
$$
(T^*M)^{1/2}\otimes (T^*M)^{1/2}\simeq T^*M. 
$$
\begin{rem}
Note that on the circle there are two non-isomorphic choices of such a root, the trivial line bundle and the M\"obius strip. In what follows we will work with the trivial root on the circle.  
\end{rem}
For an integer $\ell$ we denote by $\Omega^{\ell/2}(M)$ the space of smooth densities of weight $\ell/2$ on $M$.  That is, an element in $\Omega^{\ell/2}(M)$ is a smooth section of the $\ell$-th tensorial power of $(T^*M)^{1/2}$.} As usual, for $\ell < 0$ we define 
$$
\left((T^*M)^{1/2}\right)^{\otimes \ell}=\left((TM)^{1/2}\right)^{\otimes(-\ell)}
$$
where $(TM)^{1/2}$ denotes the dual bundle of $(T^*M)^{1/2}$. 

{Note that an affine connection $\nabla$ on $TM\simeq (T^*M)^{-1}$ induces a connection on all tensorial powers of $(T^*M)^{1/2}$. By standard abuse of notation, we will denote these connections by $\nabla$ as well. In particular, we have first order differential operators}
\begin{equation*}
 \nabla:  \Omega^{\ell/2}(M)\to \Omega^{\ell/2+1}(M).
\end{equation*}
A \emph{real projective structure}, $\mathcal{P}$ on $M$ is a second-order elliptic 
differential operator
\begin{equation*}
\mathcal{P}:\Omega^{-1/2}(M)\to \Omega^{3/2}(M)
\end{equation*}
 so that there is some affine connection $\nabla$ on $M$ and $P\in \Omega^2(M)$ with
\begin{equation*}
 \mathcal{P}=\nabla^2+P.
\end{equation*}

One verifies that, given two real projective structures $\mathcal{P}_1$ and $\mathcal{P}_2$, 
$\mathcal{P}_2-\mathcal{P}_1\in \Omega^2(M)$ is a zero-order operator.  Hence, the space of real 
projective structures is an affine space {with associated vector space} $\Omega^2(M)$.  Given an orientation preserving smooth diffeomorphism $\phi:M_1\to M_2$
 we define the push forward and pull back of real projective 
structures $\mathcal{P}_i$ on $M_i$ in an obvious fashion.  That is,
\begin{equation*}
\begin{array}{ccc}
 (\phi_* \mathcal{P}_1) \cdot \theta = (\phi^{-1})^*\left(\mathcal{P}_1 \cdot 
\phi^* \theta\right) & \mbox{and} &
 (\phi^* \mathcal{P}_2) \cdot \theta = \phi^*\left(\mathcal{P}_2 \cdot 
(\phi^{-1})^* \theta\right).
\end{array}
\end{equation*}
The \emph{Schwarzian derivative} of $\phi$ relative to $\mathcal{P}_1,\mathcal{P}_2$ is
\begin{equation*}
 S_{\mathcal{P}_1, \mathcal{P}_2}(\phi)= \phi^* \mathcal{P}_2-\mathcal{P}_1 \in \Omega^2(M_1).
\end{equation*}  
The Schwarzian satisfies the following co-cycle condition
\begin{equation} \label{cocyclecond}
 S_{\mathcal{P}_1, \mathcal{P}_3}(\phi_2 \circ \phi_1)=\phi_1^* S_{\mathcal{P}_2,\mathcal{P}_3 }(\phi_2)+ 
S_{\mathcal{P}_1,\mathcal{P}_2}(\phi_1).
\end{equation}
Given a $\phi \in \Diff_+^\infty(M)$ and a real projective structure $\mathcal{P}$  write $S_{\mathcal{P}} (\phi)=S_{\mathcal{P}, 
\mathcal{P}}(\phi)$.  An orientation preserving diffeomorphism $\phi$ is a \emph{M\"{o}bius transformation of $\mathcal{P}$} if 
and only if $S_{\mathcal{P}}(\phi)=0$. 
The co-cycle condition implies that the set of such maps forms a subgroup, $\Mob(\mathcal{P})$, of $\Diff_+^\infty(M)$. 

Let $\mathbb{RP}^1$ be the one-dimensional real projective space -- in 
other words the space of unoriented lines through the origin in $\Real^2$. Let 
$(x_1, x_2)$ be the usual linear coordinates on $\Real^2$.  If $(x_1,x_2)\neq 
0$, then we denote by $[x_1 : x_2]$ the point in $\RP^1$ corresponding to the 
line through the origin and  $(x_1,x_2)$.  On the chart $U=\set{[x_1,x_2]: 
x_2\neq 0}$ we have the affine coordinate $\tau=x_1/x_2$ for $\RP^1$.  Let 
$\taunabla$ be the (unique) connection so that $\partial_\tau$ is parallel. There is a unique real projective structure 
$\mathcal{P}_{\RP^1}$ on 
$\RP^1$ so that $\mathcal{P}_{\RP^1}=\taunabla^2$. This is the 
\emph{standard real projective structure} on $\RP^1$.   

If $\phi\in \Diff_+^\infty(\RP^1)$, then one computes that
\begin{equation*}
 S_{{\RP^1}}(\phi)=S_{\mathcal{P}_{\RP^1}}(\phi)= \left(\frac{\phi'''}{\phi'}-\frac{3}{2} \left( \frac{\phi''}{\phi'}\right)^2 
\right)\mathrm{d}\tau^2
\end{equation*}
where here $\phi'= \partial_\tau (\tau\circ \phi)$
and likewise for the higher derivatives.  
This is the classical form of the Schwarzian derivative introduced in~\S3.  Write $\Mob(\RP^1)$ for the M\"{o}bius 
group of $\mathcal{P}_{\RP^1}$ and observe these are the fractional linear transformations.
Indeed, if $\phi\in \Mob(\RP^1)$, then there is a matrix
\begin{equation*}
 L=\begin{pmatrix} a & b \\ c & d \end{pmatrix} \in \mathrm{SL}(2,\Real)
\end{equation*}
so that
\begin{equation*}
 \tau(\phi(p)) = \frac{ a \tau(p)+b}{c \tau(p)+d}.
\end{equation*}
This corresponds to the natural action of $\mathrm{SL}(2,\Real)$ on the 
space of lines through the origin. Let
\begin{equation*}
 \gamma: \mathrm{SL}(2,\Real) \to \Diff_+^\infty(\RP^1).
\end{equation*}
denote this group homomorphism.
Notice that $\ker \rho=\pm \mathrm{Id}$ and so this map induces an injective homomorphism
\begin{equation*}
 \tilde{\gamma}: \mathrm{PSL}(2,\Real) \to \Diff_+^\infty(\RP^1)
\end{equation*}
whose image is $\Mob(\RP^1)$.

Consider the natural map $T:\mathbb{S}^1 \to \RP^1$ given by sending a point $p$ to the tangent line to $\mathbb{S}^1$ through $p$. Let 
$\thetanabla$ be the unique connection on $\mathbb{S}^1$ so that $ \partial_\theta$ is parallel and let 
$\mathcal{P}_{\mathbb{S}^1}=\thetanabla^2$. 
If $\phi\in \Diff_+^\infty(\mathbb{S}^1)$, then one computes that
\begin{equation*}
 S_{\mathbb{S}^1}(\phi)=S_{\mathcal{P}_{\mathbb{S}^1}}(\phi) = \left(\frac{\phi'''}{\phi'}-\frac{3}{2} \left( 
\frac{\phi''}{\phi'}\right)^2
\right)\mathrm{d}\theta^2
\end{equation*}
where here $\phi'$ has already been defined.  Define $S_{\theta}(\phi)$ so $S_{\mathbb{S}^1}(\phi)=S_{\theta}(\phi) \mathrm{d}\theta^2.$

As $T\circ I=T$, if $\phi\in\Diff_+^\infty(\mathbb{S}^1)$ satisfies $\phi\circ I=I\circ \phi$, then there is a well-defined 
element $\tilde{T}(\phi)\in  \Diff_+^\infty(\RP^1)$ so that the following diagram is commutative:
\begin{equation*}
\begin{tikzcd}
 \mathbb{S}^1 \arrow{r}{\phi} \arrow{d}{T} & \mathbb{S}^1 \arrow{d}{T}\\
 \RP^1 \arrow{r}{\tilde{T}(\phi)} & \RP^1 
\end{tikzcd}
\end{equation*}
A straightforward computation shows that,
\begin{equation*}
 S_{\mathbb{S}^1,\RP^1}(T)=S_{\mathcal{P}_{\mathbb{S}^1},\mathcal{P}_{\RP^1}}(T)= 2 \mathrm{d}\theta^2.
\end{equation*}
Hence, for  a $\phi\in\Diff_+^\infty(\mathbb{S}^1)$ which satisfies $\phi\circ I=I\circ \phi$ the co-cycle relation for the Schwarzian 
implies
\begin{align*}
 0&=S_{\mathbb{S}^1, \RP^1}(T\circ \phi)-S_{\mathbb{S}^1, \RP^1}(\tilde{T}(\phi)\circ T) \\
  &= 2 \phi^* \mathrm{d}\theta^2 +S_{\mathbb{S}^1}(\phi)-T^* S_{\RP^1}(\tilde{T}(\phi))+2 \mathrm{d}\theta^2\\
  &=S_{\mathbb{S}^1}(\phi)+2 (\phi')^2 \mathrm{d}\theta^2 +2 \mathrm{d}\theta^2 -T^* S_{\RP^1}(\tilde{T}(\phi))
\end{align*}
That is,
\begin{equation*}
 S_{{\mathbb{S}^1}} (\phi) +2(\phi' )^2 \mathrm{d}\theta^2 -2 \mathrm{d}\theta^2=T^* S_{\RP^1} (\tilde{T}(\phi)).
\end{equation*}
One verifies from their definitions that $\tilde{T}(\Mob(\mathbb{S}^1))=\Mob(\RP^1)$ and which gives \eqref{SchwarzianMobS1eqn}.  Finally, 
we note the following commutative diagram
\begin{equation*}
 \begin{tikzcd}
  \mathrm{SL}(2,\Real) \arrow{r}{\Gamma} \arrow{d}{\pi} \arrow{rd}{\gamma}& \Mob(\mathbb{S}^1)\arrow{d}{\tilde{T}}\\
  \mathrm{PSL}(2,\Real) \arrow{r}{\tilde{\gamma}} & \Mob(\RP^1)
 \end{tikzcd}
\end{equation*}
where $\pi$ is the natural projection.

\begin{rem}
We have defined a real projective structure on $M$ is terms of a differential operator. Equivalently (and more commonly), a real projective 
structure on $M$ may be defined to be a maximal atlas mapping open sets in $M$ into $\mathbb{RP}^1$ such that the transition functions are 
restrictions of fractional linear  transformations. For the equivalency of the two definitions the reader may consult~\cite{MR2177471}.
\end{rem}

{\bf Acknowledgement.} The authors express gratitude to the anonymous referees for their careful reading and many useful suggestions.

\providecommand{\bysame}{\leavevmode\hbox to3em{\hrulefill}\thinspace}
\providecommand{\MR}{\relax\ifhmode\unskip\space\fi MR }
\providecommand{\MRhref}[2]{%
  \href{http://www.ams.org/mathscinet-getitem?mr=#1}{#2}
}
\providecommand{\href}[2]{#2}

\end{document}